\documentclass{ws-ijbc}
\usepackage{ws-rotating}     % used only when sideways tables/figures are used
\usepackage{graphicx}
\usepackage{epstopdf}
\newcommand{\Xvf}{\mathfrak{X}}

\newcommand{\xb}{{\bf x}}

\newcommand{\R}{\ensuremath{\mathbb{R}}}

\newcommand{\N}{\ensuremath{\mathbb{N}}}

\begin{document}

\catchline{}{}{}{}{} % Publisher's Area please ignore

\markboth{Y. Xia, M. Gra\v si\v c, W. Huang and V. Romanovski}{Limit Cycles  in a Model of  Olfactory Sensory Neurons}

\title{Limit Cycles  in a Model of  Olfactory Sensory Neurons}

\author{YONGHUI XIA}

\address{Department of Mathematics, Zhejiang Normal University, \\ Jinhua, 321004, P. R. China\\
xiadoc@163.com}

\author{MATEJA  GRA\v{S}I\v{C}}
\address{Faculty of Natural Science and Mathematics, University of Maribor\\
SI-2000 Maribor, Slovenia\\
Institute of Mathematics, Physics and Mechanics\\
SI-1000 Ljubljana, Slovenia\\
mateja.grasic@um.si}
\author{WENTAO HUANG\footnote{Author for correspondence}}
\address{School of Science, Guilin University of Aerospace Technology, \\ Guilin, 541004, P.R. China\\
huangwentao@163.com}
\author{VALERY  G. ROMANOVSKI}
\address{Faculty of Natural Science and Mathematics, University of Maribor\\
SI-2000 Maribor, Slovenia\\
Faculty of Electrical Engineering and Computer Science, University of  Maribor,\\
SI-2000 Maribor, Slovenia\\
Center for Applied Mathematics and Theoretical Physics\\
SI-2000 Maribor, Slovenia\\
valery.romanovsky@uni-mb.si}
\maketitle

\begin{history}
\received{(to be inserted by publisher)}
\end{history}

\begin{abstract}
We propose an approach to study small  limit cycle bifurcations on a center manifold  in analytic or smooth  systems depending
 on parameters. We then apply it to the investigation of limit cycle bifurcations in a model of
 calcium oscillations in the cilia of olfactory sensory neurons
 and show that it can have two limit cycles: a
stable cycle  appearing after a Bautin (generalized Hopf)
bifurcation and an unstable cycle appearing
after a subcritical Hopf bifurcation.
\end{abstract}


%\begin{multicols}{2}
\section{Introduction}
\noindent In this paper, we consider the mathematical model for calcium oscillations in the cilia of olfactory sensory neurons proposed in  \cite{RBS}.

The model involves three species, a cyclic-nucleotide-gated (CNG$^o$) channel, calcium $Ca^{2+}$, and calmodulin $CaM4$, which we denote by $A_1$, $A_2$ and $A_3$, respectively,
and six reactions:
\begin{equation} \label{net}
\begin{aligned}
0\xrightarrow{K_1} A_1,\ A_1\xrightarrow{K_2} A_1+A_2,\ 4A_2 \xrightarrow{K_3} A_3,\\
A_3\xrightarrow{K_4} 4A_2,\ A_1+A_3\xrightarrow{K_5} A_3,\ A_2\xrightarrow{K_6} 0,
\end{aligned}
\end{equation}
where $K_1,\dots, K_6$ are the reaction rates.
As is usual in such  schemes,  zero on the left-hand side means that the reaction is
a source, where certain substances are introduced into the system. Zero
on  the right-hand side means that the reaction  is a sink,
 where certain substances are removed
from the system.
Let us denote the concentrations of $A_1$, $A_2$ and $A_3$ by $X$, $Y$ and $Z$, respectively.

In \cite{RBS} by  using the mass action kinetics (see e.g. \cite{F})  and the generalized mass action kinetics,
the three-dimensional  system of differential equations
\begin{align} \label{sys0}
\begin{split}
      \dot{X} = &  K_1- K_5 X Z,
         \\
      \dot{Y} =  & K_2 X - 4 K_3 Y^2 +4 K_4 Z-K_6 Y^\epsilon
     \\
     \dot{Z}= & K_3 Y^2 - K_4 Z,
\end{split}
\end{align}
 associated with the network \eqref{net} is
derived.
%where it is assumed that all parameters of the system are positive.
Since the generalized mass action kinetics were used for the last reaction, the
corresponding term in the second equation is not $-K_6 Y$, but   $-K_6 Y^\epsilon$,
 where the effective exponent $\epsilon$ corresponds  to the extrusion of $Ca^{2+}$ from cilium by pumps and exchangers.

In practice, the rate constants $K_i$
are not usually known, so one of the main tasks
in the investigation of chemical reaction networks
 is  to ask whether the
resulting differential system  has  the capacity to
admit certain kinds of qualitative behavior, among them, the most
important are the behavior  near  steady states and the oscillatory
behavior.  That is, it is important to know
 whether  rate constant values such that the differential
system  resulting from a presumed chemistry admits
behavior of a specified kind can even exist. In this way, one can
determine whether a postulated chemistry taken with
mass action kinetics  can be  observed (see  \cite{F} for
more details).

%Typical biochemical  reaction models derived using the  mass action law involve polynomial differential
%equations, like the first and second  equations in \eqref{sys0}, but in some
%models the terms with non-integer exponents, like the term $-K_6 Y^\epsilon$ (where the effective exponent $\epsilon$ correspond to extrusion of $Ca^{2+}$ from cilium by pumps and exchangers), are presented.

Since  biochemical  reaction models derived using the  mass action law
are  represented by polynomial or analytical   differential equations
involving  many parameters (reaction rates),  even  the determination
of stationary states and their stability analysis become   extremely  difficult  problems, which are
 usually unfeasible for general values of  parameters, even in the case of polynomial models.
The search for  limit cycles, which describe auto-oscillatory regimes,
  is a much more difficult
problem than the investigation of  singular points.
Because of the complexity of the  biochemical reaction models,
the study of limit cycles  in such  models seldom goes
beyond the determination of possible Hopf bifurcations
(even without verifying the  transversality, or crossing, condition)
 and mostly such
bifurcations are found numerically for heuristically chosen values of parameters,
although in recent years some symbolic computation algorithms for detection of
Hopf bifurcations have been developed
(see e.g.  \cite{E,E1,NW,NW1,SturmWeberEtAl2009a}).

If a smooth system of autonomous differential  equations admits  a
two-dimensional  center manifold,
then it is possible to study not only Hopf bifurcations, but also the so-called
Bautin bifurcations,  or degenerated Hopf bifurcations on the center manifolds, see e.g.  \cite{Kuz,Farr}.
To perform the study of such bifurcations one can use one of the following six
methods known in the literature: The method of Poincar\'e-Birkhoff  normal
forms,  the method of Lyapunov quantities (constants), the method of the succession function, the
method of averaging,  the method of intrinsic harmonic balancing, and the Lyapunov-Schmidt
method (see e.g. \cite{Farr} for a nice review of the methods). Although  from a theoretical
point of view all of them allow performing a complete bifurcation analysis, in practice,
they require extremely laborious computations, so the computational efficiency
becomes an important issue. It appears that  the most efficient method
 from the computational point of view is the method
of Lyapunov quantities,  used in  \cite{BL,S}, since
it involves only collection of similar terms in polynomial
expressions and solving systems of linear algebraic equations.

%one face   hese methods allow  It appears among them t

However, the method of Lyapunov quantities used in \cite{S}
and other works involves the search for positively defined Lyapunov
functions. In this paper, we  propose a generalization of the method
for the case when the Lyapunov function is semi-positively defined,
that is, the quadratic form defined by the lowest part of the Lyapunov function
has one zero eigenvalue and the other eigenvalues are positive.
We then  apply the method to study the degenerate Hopf  bifurcations in the system \eqref{sys0} and show that the system can have two limit cycles as the result
of such bifurcations.

The paper is organized  as follows.  In Section 2, we describe  an approach
to study limit cycle bifurcations using a Lyapunov function on the center manifold.
In Section 3, we study singular points of  system \eqref{sys0}. In the last
section, we use the approach proposed in Section 2 to study limit cycles of  system \eqref{sys0}.
In particular, it is shown there that the system can have two limit cycles bifurcating from
a singular point, and numerical examples are provided confirming the existence of two
limit cycles.

%In this paper we propose an efficient computation approach  a systematic search for periodic solutions
%using  symbolic computation.

%In recent years some symbolic computation algorithms for detections possible Hopf bifurcations based on the Hurwitz??? criterion  have been developed, see,  for instance,  \cite{Erami,SturmWeberEtAl2009a,Wang}.

\section{Limit cycle bifurcations on the center manifold}
Consider a three-dimensional system of the form
\begin{align}\label{3ds}
 \dot \xb = A \xb+ F(\xb) = G(\xb),
\end{align}
 where $\xb=(x,y,z)$,  the matrix $A$ has the eigenvalues $\lambda_1, \lambda_2, \lambda_3  $
 and $ \lambda_1 <0$,  $\lambda_2= i \omega $,  $\lambda_3= -i \omega $,
  $F$ is a vector-function, which is analytic in a neighborhood of the origin,
 and such that its series expansion starts  from quadratic or higher terms, and $G(\xb)=(G_1(\xb), G_2(\xb),G_3(\xb))^T$.

Since two eigenvalues of  system \eqref{3ds}  are purely imaginary and the third one has
the real part  different from zero, according to the Center Manifold
Theorem \cite{Chi}, the system  has a center manifold   defined by a function  $x=f(y,z)$.
After a linear transformation and rescaling of time,  system \eqref{3ds} can be written in the form
\begin{equation} \label{5}
\begin{aligned}
\dot u &=      \,\;\;   -v + P(u,v,w) = \widetilde P(u,v,w) \\
\dot v &=  \,\,\;\;\;\;  u + Q(u,v,w) = \widetilde Q(u,v,s) \\
\dot w &=       -\lambda w + R(u,v,w) = \widetilde R(u,v,w),
\end{aligned}
\end{equation}
where $\lambda $ is a positive real number and $P, Q, R$, are power series  without constant and linear terms
which are convergent in a neighborhood of the origin.

% (if $\lambda$ is negative, then changing the direction of time and interchanging$u$ and $v$  we still  have a  system of the form \eqref{5} with positive $\lambda$).

%Since two eigenvalues of system \eqref{5} are pure imaginary by the Center Manifold
%Theorem \cite{?} the system  has a center manifold  $W^c$ defined by a function  $w=f(u,v)$.

%Since system \eqref{5} has a center manifold $W^c$ and $\lambda >0$,
%the trajectories in a small neighborhood of the origin tend to the trajectories
%on  the center manifold as time increases.
%By the Reduction Principle (see e.g. \cite{Bib})

%The problem of determining the
%dynamical behavior on $W^{c}$, that is, distinguishing between a center and a focus on the center manifold for a quadratic polynomial system of the form \eqref{5} was studied in \cite{Edneral}.

Since   system \eqref{5} is analytic,   for every
$r \in \N$ there exists in a sufficiently small neighborhood of the origin a $C^r$ invariant
manifold $W^c$, the local center manifold, that is tangent to the $(u,v)$-plane at the origin, and
which contains all the recurrent behavior of  system \eqref{5} in a
neighborhood of the origin in $\R^3$ (\cite[\S 4.1]{Chi}, \cite{Sij}).
 For system
\eqref{5}, the phase portrait in a neighborhood of the origin on
$W^{c}$ can be, depending on the  nonlinear terms $P$, $Q$  and $R$, either a
center, in which case every trajectory (other than the origin itself) is an
oval surrounding the origin, or a focus, in which case, every trajectory
spirals towards the origin or every trajectory spirals away from the origin
as the time increases.

According to the
Lyapunov theorem,  for
  system \eqref{5} with the  corresponding vector field
$$
\Xvf = \widetilde P \tfrac{\partial}{\partial u}
     + \widetilde Q \tfrac{\partial}{\partial v}
     + \widetilde R \tfrac{\partial}{\partial w}
$$
the origin is a center for $\Xvf | W^c$ if and only if $\Xvf$ admits a real analytic local first
integral of the form
\begin{align} \label{Int}
\Phi(u,v,w) = u^2+v^2  + \sum_{j+k+\ell = 3}^\infty \phi_{jk\ell} u^j v^k w^\ell
\end{align}
 in a neighborhood of the origin in $\R^3$.
Moreover,  when a center exists, the local center manifold $W^c$ is unique and  analytic (see  \cite[\S 13]{Bib}).

%If we are interesting in the  behavior of trajectories on the center manifold of system \eqref{5}
%we can find  an initial string of the Taylor expansion of the manifold looking
%for it in the form $w=a_1 u+ a_2 v +\dots $,
%then plug in the expansion into the first two equations of system \eqref{5}
%and then study the center focus problem for the obtained  two-dimensional system.
%However computationally more efficient way is provided by the Lyapunov Center Theorem.

For system \eqref{5},  one can  look for a function  $\Phi(u,v,w)$ of the form \eqref{Int}
such that
\begin{align} \label{ur}
%\frac{ \partial \Phi}{\partial u}\tilde P + \frac{ \partial \Phi}{\partial v} \tilde Q +\frac{ \partial \Phi}{\partial w} \tilde R =
\Xvf \Phi= \sum_{i=1}^\infty g_i (u^2+v^2)^{i+1}.
\end{align}
In the case of the two-dimensional system (when in \eqref{5} $\widetilde{R}\equiv 0$)
it is well-known  that it  is  possible to find
functions $\Phi$ and $g_i$ satisfying \eqref{ur}.
If the right hand-sides of \eqref{5} are functions depending on parameters,
then $\phi_{jk\ell}$ and $g_i$ also depend on the parameters of the system.
If, for some values of parameters, all $g_i$ vanish, then the corresponding  system \eqref{5}
 has
a center at the origin, but if, for some values of parameters, not all $g_i$ vanish,then by the
 Lyapunov stability theorem (see e.g. \cite{Bib,RS}), the singular point
at the origin is a stable focus if the first non-zero $g_i$ is negative, and it is
an unstable focus if the first non-zero $g_i$ is positive (since $\Phi$ is a  positively defined Lyapunov function
with negatively and positively defined derivatives, respectively).
If the first non-zero coefficient in \eqref{ur} is  $g_i$, then perturbing
the systen  in such way that $\vert g_{k-1}\vert \ll \vert g_k\vert$ and the signs of $g_s$ alternate
we obtain $i-1$ limit cycles bifurcated from the origin of the system \cite{S}.

%Obstacles for the fulfillment of \eqref{ur}  give us the necessary conditions for the existence of a first integral of the form \eqref{Int} for system \eqref{5}.

In the following theorem, we show that a similar approach can be applied
to study bifurcations of limit cycles on the center manifold of  three-dimensional systems
\eqref{3ds}. Although it is possible to transform  system \eqref{3ds} to a system of the form \eqref{5},
in the case when the matrix $A$ in  system \eqref{3ds} depends on parameters,
such transformation usually involves  expressions containing radicals,
so then the radicals will also appear in the coefficients of  system \eqref{5}.
It will  slow down computations of function \eqref{Int}  and focus quantities
$g_i$ in \eqref{ur} sharply. For this reason, we do not transform system \eqref{3ds} to system \eqref{5},
but work with  system \eqref{3ds}, for which we assume
that the function $G$  depends on parameters $\alpha$, $\alpha=(\alpha_1,\dots, \alpha_m).$
%Let \be \label{qfg}
%q(x) =\sum_{k+l+m=2} a_{klm} x^k y^lz^m \ee
%be a quadratic form.

\begin{theorem} \label{teo2}
Suppose  that for  system \eqref{3ds}
there exists  a polynomial
\begin{align} \label{Psi}
\Psi({\bf x}) = \sum_{j+l+m=2}^s  a_{jlm} x^j y^lz^m
\end{align}
such that
\begin{multline}
 \label{vPhi}
\Xvf(\Psi) : =  \tfrac{\partial \Psi(x)}{\partial x} G_1(\xb)
     +  \tfrac{\partial \Psi(\xb) }{\partial y}  G_2(\xb)
     +  \tfrac{\partial \Psi (\xb)}{\partial z} G_3=\\
     g_1 (y^2+z^2)^2+g_2 (y^2+ z^2)^3 +\dots+ g_{n-1} (y^2+ z^2)^n+O(||\xb||^{2n+1} ) .
\end{multline}
Let
 \begin{align} \label{xfy}
 x=f(y,z,\alpha^*)
\end{align}
  be the center manifold of  system \eqref{3ds} corresponding to the
  value  $\alpha^*$ of parameters of the system and
\begin{align} \label{qfg}
q({\bf x},\alpha^*) =\sum_{j+l+m=2} a_{jlm} x^j y^lz^m
\end{align}
be  the quadratic part of \eqref{Psi}.
Let $q_1(y,z,\alpha^*)$ be $q({\bf x},\alpha^*)$ evaluated on \eqref{xfy}.
Assume that   $q_1(y,z,\alpha^*)$
is a  positively defined quadratic form  and
\begin{align} \label{cgk}
g_1(\alpha^*)=g_2(\alpha^*)=\dots = g_k(\alpha^*)=0, \qquad g_{k+1}(\alpha^*)\neq 0,
\end{align}
where $k< n-1$.
Then, \\
 1)
If $ g_{k+1}(\alpha^*) <0$, the corresponding system \eqref{3ds} has a stable focus at the origin  on the center manifold,
and   if $ g_{k+1}(\alpha^*) >0$,  then the focus is unstable. \\
 2)
If it is possible to choose    perturbations of the parameters $\alpha$ in the system \eqref{3ds},
such that
\begin{align} \label{gper}
|g_1(\alpha_k)| \ll |g_2(\alpha_{k-1})|\ll \dots \ll |g_k(\alpha_1)|\ll  | g_{k+1}(\alpha^*)|,
\end{align}
$\alpha_{j+1}$ is arbitrarily close to $\alpha_j$ and  the signs of $g_s(\alpha_m)$ in \eqref{gper} alternate, then  system \eqref{3ds}
corresponding to the parameter  $\alpha_k$ has at least k limit cycles on the center manifold.
\end{theorem}
 \begin{proof}
1) Since
% Let  \begin{align} \label{xfy}
% x=f(y,z,\alpha^*) \end{align}
%  be the center manifold of system \eqref{s1} corresponding to
% the parameter $\alpha^*$.  Substituting \eqref{xfy} into  \eqref{Psi}
$q_1$ is positively defined,
 the function   $\Psi$ restricted to the center manifold is positively  defined
 in a small neighborhood of the origin.
 The
   derivative of $\Psi$  with respect
 to the vector field on the center manifold has the same sign as  $ g_{k+1}(\alpha^*)$.
Thus, by  the Lyapunov theorem, the  origin is a stable focus
on the center manifold if  $ g_{k+1}(\alpha^*)<0$, and an unstable focus if $ g_{k+1}(\alpha^*)>0$.

 2)
 Assume for determinacy   that  $g_{k+1}(\alpha^*)<0$.
 Under the condition of the theorem, the equality
 $ \Psi({\bf x},\alpha^*)=c$ ($c\in (0,  c_1]$) defines, in a small neighborhood of the origin
 near the center manifold \eqref{xfy},
  a family of cylinders which are
 transversal to the center manifold. Let $C_1$ be the curve formed by the intersection
 of the cylinder $\Psi ({\bf x},\alpha^*)=c_1$ and the center manifold $M(\alpha^*)$ of  system \eqref{3ds},
 defined by \eqref{xfy}. If $c_1$ is sufficiently small, then $C_1$ is
 an oval on $M(\alpha^*)$ and the vector field is directed inside $C_1$,
 %(except two points which are the intersections of $C_1$ and the plane $z=0$),
 since
 $$
 \Xvf (\Psi ({\bf x},\alpha^*))= g_{k+1}(\alpha^*) (y^2+z^2)^{k+2}+h.o.t
 $$
 and $ g_{k+1}(\alpha^*)<0.$
  By the assumption  of the theorem,
 there is an  $\alpha_1$ arbitrarily close to $\alpha^*$ and such that $g_k(\alpha_1)>0$.
 Then, for some $c_2<c_1$ the intersection of the cylinder   $ \Phi({\bf x},\alpha_1)=c_2$ ($c_2\in (0,  c_1]$)
 defines a  curve $C_2$  on the center manifold $x=f(y,z,\alpha_1)$,
such that the vector field of  system \eqref{3ds} is directed outside of $C_2$ (since $g_k(\alpha_1)>0$).
Since the perturbation is arbitrarily small, the curve $C_1$ is transformed to a curve $C_1^{(1)}$, such that
the vector field on  $C_1^{(1)}$  is still     directed inside  the curve.
Then, according to the Poincar\'e-Bendixson theorem, there is a limit cycle on the center manifold
$x=f(y,z,\alpha_1)$ in the ring bounded by $C_2$ and $C_1^{(1)}$.
Continuing the procedure on the center manifold corresponding to  a parameter $\alpha_k$
 we obtain $k$ curves $C_1^{(k)}$, $C_2^{(k-1)}, \dots, C_k$, such that
 the vector field on   $C_1^{(k)}$ is directed inside the curve,
 the vector field on  $C_2^{(k-1)}$  is directed outside of the curve,
 the vector  field on   $C_3^{(k-2)}$ is  directed inside the curve, and so on.
 Then, in each ring bounded by the curves $C^{(j)}_i$,  system \eqref{3ds} corresponding to the parameter
 $\alpha_k$ has at least one limit cycle
 on the center manifold  $x=f(y,z,\alpha_k)$.
 \end{proof}

\begin{corollary}
If  condition  \eqref{cgk} holds, then the origin of  system
\eqref{s1} is asymptotically stable if  $ g_{k+1}(\alpha^*) <0$, and it is unstable
if $ g_{k+1}(\alpha^*) >0$.
\end{corollary}
\begin{proof}
By the Reduction Principle \cite{P,GH}, the stability of the origin of a system
\eqref{s1} is the same as the stability of the singular point at the origin
on the center manifold.
\end{proof}

\section{Singular points of system \eqref{sys0}}
To simplify the study of  singular points and limit cycles of  system \eqref{sys0}, we introduce
  dimensionless variables  performing the substitution
$$
X_1 = K_2  X, \quad Y_1 =  Y, \quad Z_1= K_5 Z,
$$
which transforms \eqref{sys0} into the system

\begin{align} \label{sys}
\begin{aligned}
      \dot{X_1} = & \ k_1- X_1 Z_1  ,
         \\
      \dot{Y_1} =  & \ X_1 -4 k_3 Y_1^2 + \frac{4 k_4}{k_5} Z_1-k_2 Y_1^\epsilon
     \\
     \dot{Z_1}= & \ k_3 k_5  Y_1^2 - k_4 Z_1,
\end{aligned}
\end{align}
where $k_1=  K_1 K_2 $, $k_2= K_6,$
$k_3= K_3$,   $k_4=K_4$ and $k_5=K_5$.

Thus, without  loss of generality, instead of  system \eqref{sys0} we will study
system \eqref{sys}. Since $X$, $Y$, $Z$ in \eqref{sys0} are concentrations of the
species and $K_i$ are reaction rates, all parameters $k_i$ in \eqref{sys} are positive,
and we are interested in the behavior of trajectories of \eqref{sys}
in the domain $X_1 >0$, $Y_1 >0$, $Z_1 >0.$

%It is impossible to find stationary  points of system \eqref{sys}  for general values of parameters, so,

In order to simplify computations,
we assume that  system \eqref{sys} has a  stationary  point  in  the plane $y=1$.
It happens when
\begin{align} \label{k1}
k_1  = \frac{k_2  k_3 k_5 }{k_4},
\end{align}
 and then the unique stationary point of system \eqref{sys} is the point $P(X_1^{(0)}, Y_1^{(0)}, Z_1^{(0)})$ with the coordinates
$
 X_1^{(0)} = k_2 , \quad Y_1^{(0)}=1,\quad   Z_1^{(0)}  = \frac { k_3 k_5 }{k_4}.
$
Moving the origin to the point $P$ using the substitution
$
x=X_1- X_1^{(0)}, y= Y_1- Y_1^{(0)},  z=Z_1- Z_1^{(0)}
 $
 we obtain the system
\begin{align} \label{s1}
\begin{aligned}
      \dot{x} = &  -\frac{k_3 k_5}{k_4} x - k_2  z - x z  ,
         \\
      \dot{y} =&\  k_2  + x - 8 k_3  y -4 k_3  y^2 - k_2  (1 + y)^\epsilon
      + \frac{4 k_4}{k_5} z  ,
     \\
     \dot{z}= & \  2 k_3 k_5  y  -k_4 z   + k_3 k_5   y^2.
\end{aligned}
\end{align}
The Jacobian of the matrix of the  linear approximation of  system \eqref{s1} at the origin is
$$
A=\left(
\begin{array}{ccc}
 -\frac{k_3 k_5 }{k_4 } & 0 & -k_2
   \\
 1 & -\epsilon k_2-8 k_3 & \frac{4
   k_4 }{k_5 } \\
 0 & 2 k_3 k_5  & -k_4  \\
\end{array}
\right).
 $$
 The eigenvalues of $A$ are roots of a cubic polynomial and have rather complicated expression.
 To simplify calculations, we impose the condition that one of the eigenvalues is $-1$. To find
 this condition, we calculate the characteristic polynomial of $A$
 obtaining
 \begin{eqnarray*}
 p=
 -\frac{1}{k_4}(2 k_2  k_3 k_4 k_5  + \epsilon k_2  k_3 k_4 k_5  + \epsilon k_2  k_4^2 u +
   \epsilon k_2  k_3 k_5  u + 8 k_3^2 k_5  u + \\k_3 k_4 k_5  u + \epsilon k_2  k_4 u^2 +
   8 k_3 k_4 u^2 + k_4^2 u^2 + k_3 k_5  u^2 + k_4 u^3).
 \end{eqnarray*}
Then, the condition $p|_{u=-1}=0$ gives
 \begin{align} \label{k3}
k_2 = \frac{(-1 + 8 k_3 + k_4) (-k_4 + k_3 k_5 )}{
2 k_3 k_4 k_5  - \epsilon (-1 + k_4) (k_4 - k_3 k_5 )}.
 \end{align}

\begin{proposition} Assume that for system \eqref{s1} $\epsilon >0$ and  condition \eqref{k3} is fulfilled.
Then, the system  has a center manifold passing through the origin $O$,
with the stationary   point $O$ being a  center or a focus at the center manifold
if and only if
\begin{align}\label{cond}
k_3>0\land k_4 >0\land k_5 >0\land 8
   k_3<1\land 8 k_3+k_4 <1\land 8
   k_3 k_4 +k_3
   k_5 +k_4 ^2<k_4 .
\end{align}
\end{proposition}
\begin{proof}
Computing the eigenvalues of matrix $A$ we find that they are
$\lambda_1=-1, \ \lambda_{2,3}=\alpha \pm  \beta, $
where
\begin{equation} \label{alpha}
\begin{aligned}
\alpha= & -\frac{a}{2 k_4 (-\epsilon k_4 + \epsilon k_4^2 + \epsilon k_3 k_5  -
       2 k_3 k_4 k_5  - \epsilon k_3 k_4 k_5 )},
\end{aligned}
\end{equation}
\begin{align} \label{b}
\begin{aligned}
\beta = \frac{\sqrt{b}}{2 k_4 k_5  (-\epsilon k_4 + \epsilon k_4^2 +
     \epsilon k_3 k_5  - 2 k_3 k_4 k_5  - \epsilon k_3 k_4 k_5 )}
     \end{aligned}
\end{align}
with
$$\begin{aligned}
a =& -\epsilon k_4^3 + 8 \epsilon k_3 k_4^3 + \epsilon k_4^4 - \epsilon k_3 k_4 k_5  + 2 k_3 k_4^2 k_5  +
     2 \epsilon k_3 k_4^2 k_5  - 16 k_3^2 k_4^2 k_5  -\\& 8 \epsilon k_3^2 k_4^2 k_5  -
     2 k_3 k_4^3 k_5  - \epsilon k_3 k_4^3 k_5  + \epsilon k_3^2 k_5 ^2 - 2 k_3^2 k_4 k_5 ^2 -
      \epsilon k_3^2 k_4 k_5 ^2
\end{aligned}$$
and
$$\begin{aligned}
b= \ & k_5 ^2 (4 k_3^2 k_4^2 k_5 ^2 (2 (1 + 8 k_3) k_4^3 + k_4^4 +
          k_4^2 (-3 + 64 k_3^2 - 2 k_3 (-8 + k_5 )) +
          \\ & 2 (1 - 8 k_3) k_3 k_4 k_5  + k_3^2 k_5 ^2) +
       \epsilon^2 (k_4^4 - k_4^3 (1 + k_3 (-8 + k_5 )) - 8 k_3^2 k_4^2 k_5  - \\ &
          k_3^2 k_5 ^2 + k_3 k_4 k_5  (1 + k_3 k_5 ))^2 +
       4 \epsilon k_3 k_4 k_5  (-k_4^6 + k_3 k_4^5 (-16 + k_5 ) - k_3^3 k_5 ^3 + \\ &
          k_3^3 k_4 k_5 ^2 (8 + k_5 ) -
          k_3 (-1 + 8 k_3) k_4^2 k_5  (3 + 2 k_3 k_5 ) +
          k_4^4 (3 + 16 k_3^2 (-4 + k_5 ) +\\ & 2 k_3 k_5 ) +
          2 k_4^3 (-1 + k_3 (8 - 3 k_5 ) + 32 k_3^3 k_5  -
             k_3^2 (-8 + k_5 ) k_5 ))).
\end{aligned}$$

Thus, the matrix $A$ can have a pair of purely  imaginary eigenvalues
if and only if $\alpha=0$. Solving the latter equation for
$\epsilon $ we obtain
\begin{align} \label{ep}
\epsilon= -\frac{2 k_3 k_4 k_5  (-k_4 + 8 k_3 k_4 + k_4^2 + k_3 k_5 )}{(-k_4 + k_3 k_5 ) (-k_4^2 +
    8 k_3 k_4^2 + k_4^3 - k_3 k_5  + k_3 k_4 k_5 )}.
\end{align}
 Now, solving the semialgebraic system
$$k_1>0\land k_2>0\land k_3>0\land
   k_4 >0\land k_5 >0\land b<0\land
   \epsilon >0$$
with \texttt{Reduce}
 % \cite{Col}
  of  {\sc Mathematica}, we obtain that the solution is given
by inequalities \eqref{cond}.
\end{proof}

%{Remark}.
%% (for more details about %the model see \cite{Szed}).
%It is also possible to
%introduce the dimensionless variables performing the substitution
%$$
%X_1 = K_2 K_3 X, \quad Y_1 = 4 K_3 Y, \quad Z_1= K_5 Z,
%$$
%which transforms \eqref{sys0} into the system
%\be \label{syss}
%\begin{aligned}
%      \dot{X_1} = &  k_1- X_1 Z_1  ,
%         \\
%      \dot{Y_1} =  &  X_1 - Y_1^2 + \frac{k_4}{k_3} Z_1-k_2 Y_1^\epsilon
%     \\
%     \dot{Z_1}= & k_3 Y_1^2 - k_4 Z_1,
%\end{aligned}
%\ee
%where $k_1= 4 K_1 K_2 K_3$, $k_2= K_6 4^{1 - \epsilon} K_3^{1 - \epsilon},$
%$k_3=\frac{K_5}{16 K_3}$ and  $k_4=K_4$.
%In this case the  system has only 4 parameters, however then in a  neighborhood of
%the  singular point
%with the coordinate $y=1$ there is no oscillations.
%
\section{Limit cycles of  system \eqref{sys0}}
In this section we study the limit cycle bifurcations of  system \eqref{s1}.
%To simplify the study,  from now on we assume that for system \eqref{s1} condition \eqref{k1} is  fulfilled.
The system \eqref{s1} is not a polynomial system, so we
expand the function on the right-hand side of the second equation
of \eqref{s1} into a power series up to the third order, obtaining the  system

\begin{align} \label{s2}
\begin{aligned}
 \dot{x} = &  -\frac{k_3 k_5}{k_4} x - k_2  z - x z  ,
         \\
      \dot{y} =&\  x -(\epsilon k_2 + 8 k_3) y + \left(\frac{\epsilon k_2}2 - \frac{\epsilon^2 k_2}2 -
    4 k_3\right) y^2 \\
    & + \left(- \frac{\epsilon k_2}3 + \frac{\epsilon^2 k_2}2 -
    \frac{\epsilon^3 k_2}6\right) y^3 +
    \frac{
 4 k_4 z}{k_5} +\dots ,
     \\
     \dot{z}= & \  2 k_3 k_5  y  -k_4 z   + k_3 k_5   y^2,
\end{aligned}
\end{align}
where the dots stand for terms of the order higher than three.

Using a linear transformation, it is possible
to transform  system \eqref{sys} to a system of the form \eqref{5} (with $\lambda =-1$)
and then study its limit cycle bifurcations using the normal form theory. However, then the coefficients of the obtained system will contain radical expressions, which will essentially slow down
symbolic computations with {\sc Mathematica} (or another computer algebra system).
So, instead of transforming  \eqref{s1} to a system of the form
 \eqref{5} and then applying the normal form theory,
we use the way provided by Theorem  \ref{teo2}.
Using this approach, we look
  for  function
\eqref{Psi} satisfying \eqref{vPhi}, where now  $G_1, G_2, G_3$ are the right-hand sides of
\eqref{s2}.
% but for system \eqref{s2} we look for a  function
%\be \label{Int1}
%\Psi(u,v,w) = a x^2+ b y^2 +c y^2 +d x y +e x z +f y z
% + \sum_{j+k+\ell = 3} \psi_{jk\ell} x^j y^k z^\ell,
%\ee
%such that
%\be \label{vF}
%\Xvf(\Psi)=g_1 x^4+g_2 x^6 +\dots,
%\ee
%under the assumption
%that $\epsilon$ is defined by \eqref{ep} ($\Xvf$ in \eqref{vF} is the vector field of system \eqref{s2}).

The  computational procedure  to find the first $m$  polynomials $g_i$
is  as follows.

1.  Write down the initial string of \eqref{Psi} up to order $2m$,
$ \Psi_{2m}(x,y,z) = q({\bf x})  + \sum_{j+k+\ell = 3}^{2m} a_{jkl} x^j y^k z^l. $

2. For each $i=3,\dots, 2m$ equate coefficients of terms of  order $i$ in
the expression
\begin{align} \label{phim}
F_{2m}=\frac{ \partial \Psi_{2m}}{\partial x} G_1+ \frac{ \partial \Psi_{2m}}{\partial y}  G_2+\frac{ \partial \Psi_{2m}}{\partial z} G_3   -g_1(y^2+z^2)^2-\dots - g_{2m} (y^2+z^2)^{2m}
\end{align}
to zero, obtaining $2m-1$ systems of linear variables in unknown variables $a_{jkl}$, and
$g_1, \dots, g_{m}$.

3. Look for solutions of the obtained linear systems  consequently, starting from
systems that correspond to $i=2$. Each linear system that corresponds to odd $i=2 i_0-1$  has
a unique solution with respect to unknown  $a_{jk\ell}$. After solving the system   (for instance, with the command \texttt{Solve}
in {\sc Mathematica}),   substitute the obtained values of $a_{jkl}$ to the linear systems
that correspond to $i>2 i_0-1$. For systems that correspond to even $i=2 i_0$, consider the linear
system as a system in unknowns  $a_{jkl}$ and $g_{i_0}$. In this case, one of
 $a_{jkl}$ can be chosen arbitrarily.
 After   solving the system,     assign the value 1 to the undefined  $a_{jkl}$ if $i_0$=2, or assign the value
  0 for the  undefined  $a_{jkl}$ if $i_0>2$, then   substitute the obtained values of $a_{jkl}$ to the linear systems
that correspond to $i>2 i_0$. In this step the quantity  $g_{i_0-1}$ is computed.

The calculations using the procedure described above (we did them with  {\sc Mathematica}) yield the polynomial $g_1$ given in Appendix 1.

Let us denote  $k=(k_3,k_4,k_5)$.
%From  the Hopf theorem (see e.g. Theorem 3.4.2 in \cite{GH}) and Theorem \ref{teo2} we immediately
%have the following statement.
\begin{theorem}
If for  system \eqref{s1}  conditions \eqref{k3} and \eqref{ep}  are fulfilled, and for some fixed values $k^*=(  k_3^* , k_4^*, k_5 ^*)$ satisfying \eqref{cond}
%\be \label{cond_qv}
%k_3>0\land k_5 >0\land 8 k_3+k_4 <1\land 8 k_3 k_4 +k_3 k_5 +k_4 ^2<k_4
%\ee
  $g_1(k^*)<0$, then the system  has a stable focus at the origin  on the center manifold,
and   if $g_1(k^*)>0$, then the focus is unstable. Moreover, if
at least one of the functions  $\frac{\partial g_1(k)}{\partial k_3}, \
\frac{\partial g_1(k)}{\partial k_4}, \frac{\partial g_1(k)}{\partial k_5}$ is different from
zero for $ k=k^*$, then the system undergoes a subcritical Hopf bifurcation
if  $g_1(k^*)<0$,  and a supercritical Hopf bifurcation
if $g_1(k^*)>0$.
\end{theorem}
\begin{proof}
Calculations using the procedure described above yield that, for  system \eqref{s2},
 the quadratic part of
function \eqref{Psi} is
\begin{align} \label{qq}
 \begin{aligned}
q(x,y,z)=&-\frac{k_4}{(-1 + 8 k_3) k_4^2 + k_4^3 - k_3 k_5  + k_3 k_4 k_5 }x^2 -\\& \frac{
 2 k_3 k_5 }{(-1 + 8 k_3) k_4^2 + k_4^3 - k_3 k_5  + k_3 k_4 k_5 } x y + y^2 -\\& \frac{k_4^2 (-1 + 8 k_3 + k_4)}{ k_3 k_5  ((-1 + 8 k_3) k_4^2 + k_4^3 - k_3 k_5  + k_3 k_4 k_5 )} x z -\\& \frac{-k_4 + k_4^2 +
    k_3 k_5 }{ k_3 k_4 k_5 } y z + \frac{k_4 - (1 + 8 k_3) k_4^2 - k_3 k_5  + k_3 k_4 k_5 }{
 4 k_3^2 k_4 k_5 ^2} z^2.
\end{aligned}
\end{align}

We look for the center manifold in the form
\begin{align} \label{hcm}
x=h(y,z)
\end{align}
Then the function $h$ is computed from the  equation
$$
\dot x- \dot y \frac{\partial h}{\partial y}- \dot z \frac{\partial h}{\partial z}=0,
$$
where the left-hand side is evaluated for $x$, defined by \eqref{hcm}.

Computing the first  two  terms of the series expansion of the center manifold we find
$$
x=   \frac{k_4 -k_3 k_5 }{k_4}y -\frac{(k_4 - k_3 k_5) (k_4^2 + k_3 k_5)}{2 k_3 k_4^2 k_5}z+h.o.t.
$$
We substitute the obtained expression into \eqref{qq}, and
using the  Sylvester criterion with \texttt{Reduce} of Mathematica, verify that, if
condition \eqref{cond} holds, then the quadratic approximation of the obtained
expression is a positively defined quadratic form.
% Moreover, under this condition the
%cylinders $\Psi(x,y,z)=c$ are transversal to the plane $x=0$ for small $c$.
Thus,  from  the Hopf theorem (see e.g. Theorem 3.4.2 in \cite{GH}) and Theorem \ref{teo2} we obtain
that the conclusion of the theorem holds.
\end{proof}

Since we were unable to compute the quantity $g_2$
for general parameters $k_3$ and $k_4$ at our computational facilities,
in order to simplify the further analysis we set
\begin{align} \label{k3k4}
k_3 = k_4 = \frac{1}{10}.
\end{align}
Then, from \eqref{cond}, we obtain that  $0<k_5<0.1$.
The only root of the polynomial $g_1$ in this interval is
$\bar k_5\approx 0.05147292$,
and $g_1$ is strictly increasing on this interval. The plot of $g_1$ on this interval is given in Fig. 1.
%\begin{figure}[!h]
%\centering
%\includegraphics{g1.jpeg}
%includegraphics{cicles.jpg}
%\hspace{0.8cm}
%\includegraphics[scale=.30]{2integrs2.eps}
%\caption{...
% Trajectories of system \eqref{suv} with parameters \eqref{par_k} andthe initial conditions  $U=0$ and $V= 0.25$, $ 0.15$ and $0.07$.  }
%\label{fig1}
%\end{figure}

%We were not able to compute  the quantity  $g_2$ at our computational facilities
%for general parameters $k_3$ and $k_4$, so

The quantity $g_2$
 computed  for the values of parameters
given by \eqref{k3k4} is  given in Appendix 2.

%From the plot in Fig. 1 we see that

\begin{theorem}
There are systems \eqref{s1} with two  limit cycles in a
neighborhood  of the singular point at the origin.
\end{theorem}
\begin{proof}
Since $g_1(k_5) $ is an increasing function on $(0,0.1)$,
if  $0< k_5<\bar k_5 $,
then the singular point at the center manifold is a stable focus, and if    $\bar k_5<k_5<0.1 $,
then it is an unstable focus. If $k_5=\bar k_5$ then $g_2(\bar k_5) \approx -0.554882 <0$ and, therefore, the singular
point is a stable focus. Thus, according to Theorem \ref{teo2},  after the perturbation of $k_5$  in a neighborhood of
$k_5=\bar k_5$
in such a way that $g_1$ becomes positive,  a stable limit cycle is born at the center manifold.

Since, after such perturbation, the real part of the eigenvalue is still zero and
the transversality condition for $\alpha$ defined by \eqref{alpha}  is satisfied, one more limit cycle is born as
the result of the Hopf bifurcation.
\end{proof}

\begin{figure}[h]
\centering
\includegraphics[scale=0.35]{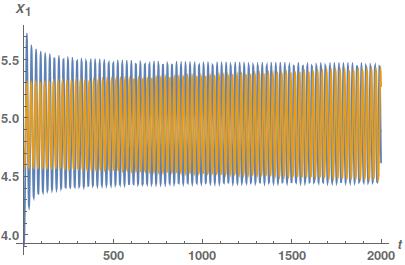}
\hspace{0.5cm}
\includegraphics[scale=0.35]{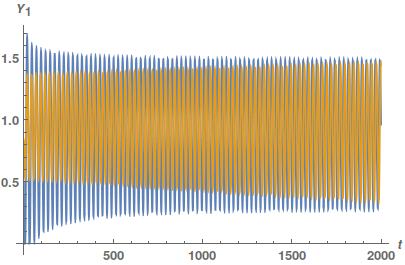}
\hspace{0.5cm}
\includegraphics[scale=0.35]{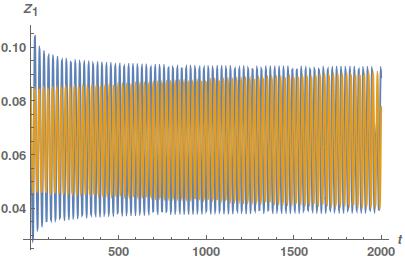}
\caption{Solutions $X_1(t)$, $Y_1(t)$ and $Z_1(t)$ of the initial problems $X_1(0)=3.9267, Y_1(0)=0.5, Z_1(0)=0.085$  and $X_1(0)=4.7267, Y_1(0)=0.8, Z_1(0)=0.085$ of the system \eqref{sys} for parameters $k^*$ given
in \eqref{k*} converge to a stable limit cycle when $t \rightarrow \infty$.}
\label{fig1}
\end{figure}

Example 1.
The existence of a stable  limit cycle of  system \eqref{sys}
for the values of parameters { \begin{equation} \label{k*}
k^*= (k_1 , k_2 , k_3, k_4, k_5 , \epsilon)=(0.320238, 4.92673, 0.1, 0.1, 0.065, 0.0071041)
\end{equation}}
is evident from Figure \ref{fig1}. Since for these values of parameters the real part of the eigenvalues
of the singular point at the origin of system \eqref{s1} is zero
and
$$
\frac{\partial \alpha}{\partial k_4}|_{k^*}\ne 0,
$$
where $\alpha$ is defined by \eqref{alpha}, an unstable limit cycle
can  appear from the singular points after the  Hopf bifurcation.

%\begin{figure}[!h]\centering
%\includegraphics[scale=.65]{fig2y.jpg}
%includegraphics{cicles.jpg}
%\hspace{0.8cm}
%\includegraphics[scale=.30]{2integrs2.eps}
%\caption{Graph of solutions $Y_1(t)$ of the system \eqref{sys} for parameters $k^*$ in \eqref{k*} and the initial conditions $(X_1, Y_1, Z_1)=( , , )$ and $( , , )$.
%Trajectories of system \eqref{suv} with parameters \eqref{par_k} andthe initial conditions  $U=0$ and $V= 0.25$, $ 0.15$ and $0.07$.
% }\label{fig2}
%\end{figure}

Example 2. In Figure \ref{fig3} we observe 2 limit cycles in  system \eqref{sys}
for the values of parameters
\begin{equation}\label{k**}
k^*= \{k_1 , k_2 , k_3, k_4, k_5 , \epsilon\}=\left\{0.320238, 4.92673, 0.1, 0.1, 0.065, 0.0072041\right\}
.
\end{equation}
The outer stable limit cycle is clearly visible.
Since for these values of parameters $\alpha$ defined by \eqref{alpha}
is negative, the singular point is stable.
Thus,  the trajectories
corresponding to the smallest ring in Figure \ref{fig3} tend to zero when time increases. It means that there is
an unstable limit cycle in the area between the smallest and middle rings in Figure \ref{fig3}.

\begin{figure}[h]
\centering
\includegraphics[scale=0.75]{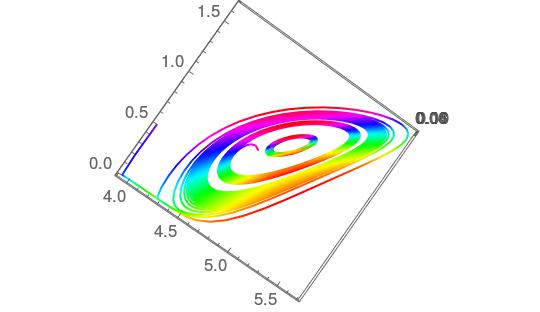}
%includegraphics{cicles.jpg}
%\hspace{0.8cm}
%\includegraphics[scale=.30]{2integrs2.eps}
\caption{
Two limit cycles in system  \eqref{sys} with parameters \eqref{k**}
and the initial conditions  $(X_1, Y_1, Z_1)=( , , )$, $( , , )$ and $( , , )$.
}
\label{fig3}
\end{figure}

To conclude, we have shown that for some values of parameters in  system
\eqref{sys0} not only a Hopf bifurcation occurs, but also degenerate
Hopf bifurcations occur, so there are systems in the family with two limit cycles.

\nonumsection{Acknowledgments} \noindent
Y. Xia was supported by the National Natural Science Foundation of China (grants No. 11671176 and No. 11271333), Natural Science Foundation of Zhejiang Province (grant No. LY15A010007),  the Scientific Research Funds of Huaqiao University and Natural Science Foundation of Fujian Province (grant No. 2018J01001), M. Gra\v si\v c by the Slovenian Research Agency  (program P1-0288, project J1-8133), W. Huang
by  the Natural Science Foundation of Guangxi, China (grant No. 2016GXNSFDA380031), V. G.  Romanovski   by the Slovenian Research Agency  (program P1-0306, project N1-0063). We thank Dr. Andre Zegeling for helpful remarks and  fruitful
discussions.

 \section*{Appendix 1}
{{
$
g_1 = -((-k_4^3 + 8 k_3 k_4^3 + k_4^4 - k_3 k_4 k_5  + k_3 k_4^2 k_5  +
        k_3^2 k_5 ^2) (-3 k_4^{14} + 288 k_3 k_4^{14} - 8320 k_3^2 k_4^{14} +
        112640 k_3^3 k_4^{14} - 798720 k_3^4 k_4^{14} + 2883584 k_3^5 k_4^{14} -
        4194304 k_3^6 k_4^{14} + 68 k_4^{15} - 3448 k_3 k_4^{15} +
        67904 k_3^2 k_4^{15} - 670208 k_3^3 k_4^{15} + 3485696 k_3^4 k_4^{15} -
        8912896 k_3^5 k_4^{15} + 8388608 k_3^6 k_4^{15} - 258 k_4^{16} +
        9992 k_3 k_4^{16} - 149248 k_3^2 k_4^{16} + 1073664 k_3^3 k_4^{16} -
        3702784 k_3^4 k_4^{16} + 4849664 k_3^5 k_4^{16}+ 416 k_4^{17}-
        12424 k_3 k_4^{17}+ 136000 k_3^2 k_4^{17}- 645120 k_3^3 k_4^{17}+
        1114112 k_3^4 k_4^{17}- 335 k_4^{18}+ 7384 k_3 k_4^{18}-
        53504 k_3^2 k_4^{18}+ 126976 k_3^3 k_4^{18}+ 132 k_4^{19}-
        1952 k_3 k_4^{19}+ 7168 k_3^2 k_4^{19}- 20 k_4^{20}+ 160 k_3 k_4^{20}-
        12 k_3 k_4^{12}k_5  + 1056 k_3^2 k_4^{12}k_5  - 25344 k_3^3 k_4^{12}k_5  +
        264192 k_3^4 k_4^{12}k_5  - 1277952 k_3^5 k_4^{12}k_5  +
        2359296 k_3^6 k_4^{12}k_5  + 342 k_3 k_4^{13}k_5  -
        15640 k_3^2 k_4^{13}k_5  + 270592 k_3^3 k_4^{13}k_5  -
        2287104 k_3^4 k_4^{13}k_5  + 9928704 k_3^5 k_4^{13}k_5  -
        20840448 k_3^6 k_4^{13}k_5  + 16777216 k_3^7 k_4^{13}k_5  -
        1679 k_3 k_4^{14}k_5  + 61408 k_3^2 k_4^{14}k_5  -
        872704 k_3^3 k_4^{14}k_5  + 6153216 k_3^4 k_4^{14}k_5  -
        22704128 k_3^5 k_4^{14}k_5  + 42336256 k_3^6 k_4^{14}k_5  -
        33554432 k_3^7 k_4^{14}k_5  + 3618 k_3 k_4^{15}k_5  -
        108416 k_3^2 k_4^{15}k_5  + 1242752 k_3^3 k_4^{15}k_5  -
        6830592 k_3^4 k_4^{15}k_5  + 18251776 k_3^5 k_4^{15}k_5  -
        19791872 k_3^6 k_4^{15}k_5  - 4135 k_3 k_4^{16}k_5  +
        98864 k_3^2 k_4^{16}k_5  - 860736 k_3^3 k_4^{16}k_5  +
        3260416 k_3^4 k_4^{16}k_5  - 4653056 k_3^5 k_4^{16}k_5  +
        2606 k_3 k_4^{17}k_5  - 46936 k_3^2 k_4^{17}k_5  +
        276992 k_3^3 k_4^{17}k_5  - 544768 k_3^4 k_4^{17}k_5  -
        852 k_3 k_4^{18}k_5  + 10400 k_3^2 k_4^{18}k_5  - 31744 k_3^3 k_4^{18}k_5  +
        112 k_3 k_4^{19}k_5  - 736 k_3^2 k_4^{19}k_5  - 18 k_3^2 k_4^{10}k_5 ^2 +
        1440 k_3^3 k_4^{10}k_5 ^2 - 27264 k_3^4 k_4^{10}k_5 ^2 +
        196608 k_3^5 k_4^{10}k_5 ^2 - 491520 k_3^6 k_4^{10}k_5 ^2 +
        690 k_3^2 k_4^{11}k_5 ^2 - 27624 k_3^3 k_4^{11}k_5 ^2 +
        400064 k_3^4 k_4^{11}k_5 ^2 - 2664448 k_3^5 k_4^{11}k_5 ^2 +
        8208384 k_3^6 k_4^{11}k_5 ^2 - 9437184 k_3^7 k_4^{11}k_5 ^2 -
        4370 k_3^2 k_4^{12}k_5 ^2 + 144784 k_3^3 k_4^{12}k_5 ^2 -
        1830336 k_3^4 k_4^{12}k_5 ^2 + 11218432 k_3^5 k_4^{12}k_5 ^2 -
        34574336 k_3^6 k_4^{12}k_5 ^2 + 48234496 k_3^7 k_4^{12}k_5 ^2 -
        16777216 k_3^8 k_4^{12}k_5 ^2 + 12134 k_3^2 k_4^{13}k_5 ^2 -
        342592 k_3^3 k_4^{13}k_5 ^2 + 3705152 k_3^4 k_4^{13}k_5 ^2 -
        19392000 k_3^5 k_4^{13}k_5 ^2 + 50798592 k_3^6 k_4^{13}k_5 ^2 -
        61997056 k_3^7 k_4^{13}k_5 ^2 + 33554432 k_3^8 k_4^{13}k_5 ^2 -
        18133 k_3^2 k_4^{14}k_5 ^2 + 426432 k_3^3 k_4^{14}k_5 ^2 -
        3743232 k_3^4 k_4^{14}k_5 ^2 + 15162368 k_3^5 k_4^{14}k_5 ^2 -
        28049408 k_3^6 k_4^{14}k_5 ^2 + 19922944 k_3^7 k_4^{14}k_5 ^2 +
        15526 k_3^2 k_4^{15}k_5 ^2 - 290456 k_3^3 k_4^{15}k_5 ^2 +
        1907840 k_3^4 k_4^{15}k_5 ^2 - 5130240 k_3^5 k_4^{15}k_5 ^2 +
        4718592 k_3^6 k_4^{15}k_5 ^2 - 7541 k_3^2 k_4^{16}k_5 ^2 +
        104208 k_3^3 k_4^{16}k_5 ^2 - 443648 k_3^4 k_4^{16}k_5 ^2 +
        557056 k_3^5 k_4^{16}k_5 ^2 + 1900 k_3^2 k_4^{17}k_5 ^2 -
        16960 k_3^3 k_4^{17}k_5 ^2 + 32768 k_3^4 k_4^{17}k_5 ^2 -
        188 k_3^2 k_4^{18}k_5 ^2 + 768 k_3^3 k_4^{18}k_5 ^2 -
        12 k_3^3 k_4^8 k_5 ^3 + 864 k_3^4 k_4^8 k_5 ^3 -
        11776 k_3^5 k_4^8 k_5 ^3 + 45056 k_3^6 k_4^8 k_5 ^3 +
        694 k_3^3 k_4^9 k_5 ^3 - 23400 k_3^4 k_4^9 k_5 ^3 +
        264960 k_3^5 k_4^9 k_5 ^3 - 1218560 k_3^6 k_4^9 k_5 ^3 +
        1933312 k_3^7 k_4^9 k_5 ^3 - 5874 k_3^3 k_4^{10}k_5 ^3 +
        169888 k_3^4 k_4^{10}k_5 ^3 - 1800832 k_3^5 k_4^{10}k_5 ^3 +
        8670208 k_3^6 k_4^{10}k_5 ^3 - 18333696 k_3^7 k_4^{10}k_5 ^3 +
        11534336 k_3^8 k_4^{10}k_5 ^3 + 21136 k_3^3 k_4^{11}k_5 ^3 -
        539104 k_3^4 k_4^{11}k_5 ^3 + 5157888 k_3^5 k_4^{11}k_5 ^3 -
        23117824 k_3^6 k_4^{11}k_5 ^3 + 48726016 k_3^7 k_4^{11}k_5 ^3 -
        40370176 k_3^8 k_4^{11}k_5 ^3 - 40377 k_3^3 k_4^{12}k_5 ^3 +
        885440 k_3^4 k_4^{12}k_5 ^3 - 7203712 k_3^5 k_4^{12}k_5 ^3 +
        26811392 k_3^6 k_4^{12}k_5 ^3 - 44810240 k_3^7 k_4^{12}k_5 ^3 +
        26214400 k_3^8 k_4^{12}k_5 ^3 + 44336 k_3^3 k_4^{13}k_5 ^3 -
        798872 k_3^4 k_4^{13}k_5 ^3 + 5100160 k_3^5 k_4^{13}k_5 ^3 -
        13600768 k_3^6 k_4^{13}k_5 ^3 + 12976128 k_3^7 k_4^{13}k_5 ^3 -
        28329 k_3^3 k_4^{14}k_5 ^3 + 391376 k_3^4 k_4^{14}k_5 ^3 -
        1725440 k_3^5 k_4^{14}k_5 ^3 + 2457600 k_3^6 k_4^{14}k_5 ^3 +
        10054 k_3^3 k_4^{15}k_5 ^3 - 94624 k_3^4 k_4^{15}k_5 ^3 +
        217088 k_3^5 k_4^{15}k_5 ^3 - 1724 k_3^3 k_4^{16}k_5 ^3 +
        8448 k_3^4 k_4^{16}k_5 ^3 + 96 k_3^3 k_4^{17}k_5 ^3 -
        3 k_3^4 k_4^6 k_5 ^4 + 192 k_3^5 k_4^6 k_5 ^4 - 1536 k_3^6 k_4^6 k_5 ^4 +
        346 k_3^4 k_4^7 k_5 ^4 - 9312 k_3^5 k_4^7 k_5 ^4 +
        73472 k_3^6 k_4^7 k_5 ^4 - 172032 k_3^7 k_4^7 k_5 ^4 -
        4322 k_3^4 k_4^8 k_5 ^4 + 104136 k_3^5 k_4^8 k_5 ^4 -
        853440 k_3^6 k_4^8 k_5 ^4 + 2750464 k_3^7 k_4^8 k_5 ^4 -
        2621440 k_3^8 k_4^8 k_5 ^4 + 20944 k_3^4 k_4^9 k_5 ^4 -
        463160 k_3^5 k_4^9 k_5 ^4 + 3682688 k_3^6 k_4^9 k_5 ^4 -
        12698624 k_3^7 k_4^9 k_5 ^4 + 17268736 k_3^8 k_4^9 k_5 ^4 -
        4194304 k_3^9 k_4^9 k_5 ^4 - 51788 k_3^4 k_4^{10}k_5 ^4 +
        1019000 k_3^5 k_4^{10}k_5 ^4 - 7254080 k_3^6 k_4^{10}k_5 ^4 +
        22577152 k_3^7 k_4^{10}k_5 ^4 - 28655616 k_3^8 k_4^{10}k_5 ^4 +
        8388608 k_3^9 k_4^{10}k_5 ^4 + 72244 k_3^4 k_4^{11}k_5 ^4 -
        1201240 k_3^5 k_4^{11}k_5 ^4 + 6972160 k_3^6 k_4^{11}k_5 ^4 -
        16416768 k_3^7 k_4^{11}k_5 ^4 + 12451840 k_3^8 k_4^{11}k_5 ^4 -
        58443 k_3^4 k_4^{12}k_5 ^4 + 763216 k_3^5 k_4^{12}k_5 ^4 -
        3156992 k_3^6 k_4^{12}k_5 ^4 + 4063232 k_3^7 k_4^{12}k_5 ^4 +
        26670 k_3^4 k_4^{13}k_5 ^4 - 242592 k_3^5 k_4^{13}k_5 ^4 +
        532480 k_3^6 k_4^{13}k_5 ^4 - 6192 k_3^4 k_4^{14}k_5 ^4 +
        29696 k_3^5 k_4^{14}k_5 ^4 + 544 k_3^4 k_4^{15}k_5 ^4 +
        68 k_3^5 k_4^5 k_5 ^5 - 1344 k_3^6 k_4^5 k_5 ^5 +
        5632 k_3^7 k_4^5 k_5 ^5 - 1655 k_3^5 k_4^6 k_5 ^5 +
        30784 k_3^6 k_4^6 k_5 ^5 - 168192 k_3^7 k_4^6 k_5 ^5 +
        241664 k_3^8 k_4^6 k_5 ^5 + 11838 k_3^5 k_4^7 k_5 ^5 -
        215104 k_3^6 k_4^7 k_5 ^5 + 1296256 k_3^7 k_4^7 k_5 ^5 -
        2836480 k_3^8 k_4^7 k_5 ^5 + 1441792 k_3^9 k_4^7 k_5 ^5 -
        39681 k_3^5 k_4^8 k_5 ^5 + 671152 k_3^6 k_4^8 k_5 ^5 -
        3906240 k_3^7 k_4^8 k_5 ^5 + 9007104 k_3^8 k_4^8 k_5 ^5 -
        6291456 k_3^9 k_4^8 k_5 ^5 + 71812 k_3^5 k_4^9 k_5 ^5 -
        1062240 k_3^6 k_4^9 k_5 ^5 + 5312896 k_3^7 k_4^9 k_5 ^5 -
        10094592 k_3^8 k_4^9 k_5 ^5 + 5505024 k_3^9 k_4^9 k_5 ^5 -
        73618 k_3^5 k_4^{10}k_5 ^5 + 878128 k_3^6 k_4^{10}k_5 ^5 -
        3235072 k_3^7 k_4^{10}k_5 ^5 + 3538944 k_3^8 k_4^{10}k_5 ^5 +
        42216 k_3^5 k_4^{11}k_5 ^5 - 356768 k_3^6 k_4^{11}k_5 ^5 +
        708608 k_3^7 k_4^{11}k_5 ^5 - 12388 k_3^5 k_4^{12}k_5 ^5 +
        55296 k_3^6 k_4^{12}k_5 ^5 + 1408 k_3^5 k_4^{13}k_5 ^5 -
        258 k_3^6 k_4^4 k_5 ^6 + 3168 k_3^7 k_4^4 k_5 ^6 -
        7680 k_3^8 k_4^4 k_5 ^6 + 3514 k_3^6 k_4^5 k_5 ^6 -
        47840 k_3^7 k_4^5 k_5 ^6 + 179456 k_3^8 k_4^5 k_5 ^6 -
        147456 k_3^9 k_4^5 k_5 ^6 - 17641 k_3^6 k_4^6 k_5 ^6 +
        239952 k_3^7 k_4^6 k_5 ^6 - 1016896 k_3^8 k_4^6 k_5 ^6 +
        1327104 k_3^9 k_4^6 k_5 ^6 - 262144 k_3^{10}k_4^6 k_5 ^6 +
        43696 k_3^6 k_4^7 k_5 ^6 - 548304 k_3^7 k_4^7 k_5 ^6 +
        2182528 k_3^8 k_4^7 k_5 ^6 - 2856960 k_3^9 k_4^7 k_5 ^6 +
        524288 k_3^{10}k_4^7 k_5 ^6 - 58519 k_3^6 k_4^8 k_5 ^6 +
        614848 k_3^7 k_4^8 k_5 ^6 - 1908736 k_3^8 k_4^8 k_5 ^6 +
        1622016 k_3^9 k_4^8 k_5 ^6 + 42672 k_3^6 k_4^9 k_5 ^6 -
        325824 k_3^7 k_4^9 k_5 ^6 + 569344 k_3^8 k_4^9 k_5 ^6 -
        15704 k_3^6 k_4^{10}k_5 ^6 + 64512 k_3^7 k_4^{10}k_5 ^6 +
        2240 k_3^6 k_4^{11}k_5 ^6 + 416 k_3^7 k_4^3 k_5 ^7 -
        3392 k_3^8 k_4^3 k_5 ^7 + 4608 k_3^9 k_4^3 k_5 ^7 -
        4039 k_3^7 k_4^4 k_5 ^7 + 39008 k_3^8 k_4^4 k_5 ^7 -
        89856 k_3^9 k_4^4 k_5 ^7 + 32768 k_3^{10}k_4^4 k_5 ^7 +
        15354 k_3^7 k_4^5 k_5 ^7 - 149840 k_3^8 k_4^5 k_5 ^7 +
        397568 k_3^9 k_4^5 k_5 ^7 - 221184 k_3^{10}k_4^5 k_5 ^7 -
        28653 k_3^7 k_4^6 k_5 ^7 + 249568 k_3^8 k_4^6 k_5 ^7 -
        588544 k_3^9 k_4^6 k_5 ^7 + 278528 k_3^{10}k_4^6 k_5 ^7 +
        27662 k_3^7 k_4^7 k_5 ^7 - 183712 k_3^8 k_4^7 k_5 ^7 +
        260096 k_3^9 k_4^7 k_5 ^7 - 13108 k_3^7 k_4^8 k_5 ^7 +
        48128 k_3^8 k_4^8 k_5 ^7 + 2368 k_3^7 k_4^9 k_5 ^7 -
        335 k_3^8 k_4^2 k_5 ^8 + 1696 k_3^9 k_4^2 k_5 ^8 -
        1024 k_3^{10}k_4^2 k_5 ^8 + 2642 k_3^8 k_4^3 k_5 ^8 -
        16256 k_3^9 k_4^3 k_5 ^8 + 16896 k_3^{10}k_4^3 k_5 ^8 -
        7849 k_3^8 k_4^4 k_5 ^8 + 49792 k_3^9 k_4^4 k_5 ^8 -
        61184 k_3^{10}k_4^4 k_5 ^8 + 10886 k_3^8 k_4^5 k_5 ^8 -
        57440 k_3^9 k_4^5 k_5 ^8 + 55296 k_3^{10}k_4^5 k_5 ^8 -
        7008 k_3^8 k_4^6 k_5 ^8 + 21760 k_3^9 k_4^6 k_5 ^8 +
        1664 k_3^8 k_4^7 k_5 ^8 + 132 k_3^9 k_4 k_5 ^9 - 320 k_3^{10}k_4 k_5 ^9 -
        940 k_3^9 k_4^2 k_5 ^9 + 2752 k_3^{10}k_4^2 k_5 ^9 +
        2228 k_3^9 k_4^3 k_5 ^9 - 6976 k_3^{10}k_4^3 k_5 ^9 -
        2156 k_3^9 k_4^4 k_5 ^9 + 4864 k_3^{10}k_4^4 k_5 ^9 +
        736 k_3^9 k_4^5 k_5 ^9 - 20 k_3^{10}k_5 ^{10}+ 144 k_3^{10}k_4 k_5 ^{10}-
        284 k_3^{10}k_4^2 k_5 ^{10}+ 160 k_3^{10}k_4^3 k_5 ^{10}))/(k_4^2 (-k_4 +
        k_3 k_5 ) (-k_4^2 + 8 k_3 k_4^2 + k_4^3 - k_3 k_5  +
        k_3 k_4 k_5 )^3 (-k_4^2 - 4 k_4^3 + 32 k_3 k_4^3 + 4 k_4^4 -
        4 k_3 k_4 k_5  + 4 k_3 k_4^2 k_5  + 4 k_3^2 k_5 ^2) (3 k_4^6 -
        48 k_3 k_4^6 + 192 k_3^2 k_4^6 + 6 k_3 k_4^4 k_5  - 48 k_3^2 k_4^4 k_5  -
        6 k_3 k_4^5 k_5  + 48 k_3^2 k_4^5 k_5  + 3 k_3^2 k_4^2 k_5 ^2 -
        12 k_3^2 k_4^3 k_5 ^2 + 48 k_3^3 k_4^3 k_5 ^2 + 3 k_3^2 k_4^4 k_5 ^2 +
        8 k_3^2 k_4^5 k_5 ^2 - 64 k_3^3 k_4^5 k_5 ^2 + 16 k_3^2 k_4^6 k_5 ^2 -
        6 k_3^3 k_4 k_5 ^3 + 6 k_3^3 k_4^2 k_5 ^3 + 8 k_3^3 k_4^3 k_5 ^3 -
        8 k_3^3 k_4^4 k_5 ^3 + 3 k_3^4 k_5 ^4 - 8 k_3^4 k_4^2 k_5 ^4 +
        48 k_3^4 k_4^4 k_5 ^4))
$}}
 \section*{Appendix 2}
{{
$\begin{aligned}
g_2 = &- \big(2936684087935203 + 1903329794381968962 k_5  +
       526356269068940036058 k_5 ^2 + \\
      & 79212661161468359120682 k_5 ^3 +
       6624774681223327162056312 k_5 ^4 + \\
      & 226210368780713535116189198 k_5 ^5 -
       11069401291544927561200467090 k_5 ^6 - \\
      & 1681061474736283778980266837290 k_5 ^7 -
       88270101979612138277063002374075 k_5 ^8 - \\
      & 2477699412162830370947947313822600 k_5 ^9 -
       36765039251191521107996938881205000 k_5 ^{10}- \\
      & 225936898431835263136354705676085000 k_5 ^{11}+
       135558039338577164252680767477500000 k_5 ^{12}+ \\
      & 4225634014554480900089535196731000000 k_5 ^{13}-
       7447510173322764423476685804685000000 k_5 ^{14}- \\
      & 55286614599318489486752585449650000000 k_5 ^{15}+
       251173978392196043283302794468500000000 k_5 ^{16}- \\
      & 307896287935124803285928790000000000 k_5 ^{17}-
       2169631665992185188025531440150000000000 k_5 ^{18}+ \\
      & 5262165240502910370909490745000000000000 k_5 ^{19}-
       2123588888489564564273413665000000000000 k_5 ^{20}- \\
      & 13651336452235596514164796000000000000000 k_5 ^{21}+
       35407308409324280850690466000000000000000 k_5 ^{22}- \\
      & 47646890714938134580571205000000000000000 k_5 ^{23} +
       43901452243104299918553400000000000000000 k_5 ^{24} - \\
      & 31151597239468216641843000000000000000000 k_5 ^{25} +
       18249194021574604996285000000000000000000 k_5 ^{26} - \\
      & 8792944539548463715050000000000000000000 k_5 ^{27} +
       3137409734404532590500000000000000000000 k_5 ^{28} - \\
      & 621925208863774160000000000000000000000 k_5 ^{29} -
       40840825018999950000000000000000000000 k_5 ^{30} + \\
      & 64682248399281000000000000000000000000 k_5 ^{31} -
       17749240540020000000000000000000000000 k_5 ^{32} + \\
      & 2159429746500000000000000000000000000 k_5 ^{33} -
       94431744000000000000000000000000000 k_5 ^{34}\big)/ \\
      & \big(36000 (-1 +
         k_5 )^3 (1 + 90 k_5 )^5 (-13 - 45 k_5  + 50 k_5 ^2)^2 (-401 -
         90 k_5  + 100 k_5 ^2)\\
      & (98209 + 739620 k_5  - 748900 k_5 ^2 -
         162000 k_5 ^3 + 90000 k_5 ^4) \\
      & (-5 - 675 k_5  - 29643 k_5 ^2 -
         389205 k_5 ^3 + 1470264 k_5 ^4 - 1661040 k_5 ^5 +
         610560 k_5 ^6) \\
      & (3 + 540 k_5  + 30683 k_5 ^2 + 561150 k_5 ^3 -
         34888 k_5 ^4 - 2005920 k_5 ^5 + 1462400 k_5 ^6)\big)
\end{aligned}$ }}

%\end{multicols}
\end{document}